\documentclass[12pt]{amsart}
\usepackage{amsmath,amssymb,amsbsy,amsfonts,latexsym,amsopn,amstext,
                                               amsxtra,euscript,amscd,bm}
                    
\usepackage{url}
\usepackage[colorlinks,linkcolor=blue,anchorcolor=blue,citecolor=blue]{hyperref}
\usepackage{cleveref}
\usepackage{color}
\numberwithin{equation}{section}
\usepackage[left=70pt, right=70pt]{geometry}

\begin{document}

\newtheorem{problem}{Problem}
\newtheorem{theorem}{Theorem}[section]
\newtheorem{lemma}[theorem]{Lemma}
\newtheorem{corollary}[theorem]{Corollary}
\newtheorem{example}[theorem]{Example}
\newtheorem{claim}[theorem]{Claim}
\newtheorem{cor}[theorem]{Corollary}
\newtheorem{prop}[theorem]{Proposition}
\newtheorem{definition}[theorem]{Definition}
\newtheorem{question}[theorem]{Question}
\newtheorem{conj}{Conjecture}
\newtheorem{hypothesis}{Hypothesis}
\def\cA{{\mathcal A}}
\def\cB{{\mathcal B}}
\def\cC{{\mathcal C}}
\def\cD{{\mathcal D}}
\def\cE{{\mathcal E}}
\def\cF{{\mathcal F}}
\def\cG{{\mathcal G}}
\def\cH{{\mathcal H}}
\def\cI{{\mathcal I}}
\def\cJ{{\mathcal J}}
\def\cK{{\mathcal K}}
\def\cL{{\mathcal L}}
\def\cM{{\mathcal M}}
\def\cN{{\mathcal N}}
\def\cO{{\mathcal O}}
\def\cP{{\mathcal P}}
\def\cQ{{\mathcal Q}}
\def\cR{{\mathcal R}}
\def\cS{{\mathcal S}}
\def\cT{{\mathcal T}}
\def\cU{{\mathcal U}}
\def\cV{{\mathcal V}}
\def\cW{{\mathcal W}}
\def\cX{{\mathcal X}}
\def\cY{{\mathcal Y}}
\def\cZ{{\mathcal Z}}

\def\A{{\mathbb A}}
\def\B{{\mathbb B}}
\def\C{{\mathbb C}}
\def\D{{\mathbb D}}
\def\E{{\mathbb E}}
\def\F{{\mathbb F}}
\def\G{{\mathbb G}}
\def\I{{\mathbb I}}
\def\J{{\mathbb J}}
\def\K{{\mathbb K}}
\def\L{{\mathbb L}}
\def\M{{\mathbb M}}
\def\N{{\mathbb N}}
\def\O{{\mathbb O}}
\def\P{{\mathbb P}}
\def\Q{{\mathbb Q}}
\def\R{{\mathbb R}}
\def\S{{\mathbb S}}
\def\T{{\mathbb T}}
\def\U{{\mathbb U}}
\def\V{{\mathbb V}}
\def\W{{\mathbb W}}
\def\X{{\mathbb X}}
\def\Y{{\mathbb Y}}
\def\Z{{\mathbb Z}}

\def\e{{\mathbf{e}}}
\def\ep{{\mathbf{e}}_p}
\def\eq{{\mathbf{e}}_q}

\def\scr{\scriptstyle}
\def\({\left(}
\def\){\right)}
\def\[{\left[}
\def\]{\right]}
\def\<{\langle}
\def\>{\rangle}
\def\fl#1{\left\lfloor#1\right\rfloor}
\def\rf#1{\left\lceil#1\right\rceil}
\def\le{\leqslant}
\def\ge{\geqslant}
\def\eps{\varepsilon}
\def\mand{\qquad\mbox{and}\qquad}

\def\sssum{\mathop{\sum\ \sum\ \sum}}
\def\ssum{\mathop{\sum\, \sum}}
\def\ssumw{\mathop{\sum\qquad \sum}}

\def\vec#1{\mathbf{#1}}
\def\inv#1{\overline{#1}}
\def\num#1{\mathrm{num}(#1)}
\def\dist{\mathrm{dist}}

\def\fA{{\mathfrak A}}
\def\fB{{\mathfrak B}}
\def\fC{{\mathfrak C}}
\def\fU{{\mathfrak U}}
\def\fV{{\mathfrak V}}

\newcommand{\bflambda}{{\boldsymbol{\lambda}}}
\newcommand{\bfxi}{{\boldsymbol{\xi}}}
\newcommand{\bfrho}{{\boldsymbol{\rho}}}
\newcommand{\bfnu}{{\boldsymbol{\nu}}}
\newcommand*\diff{\mathop{}\!\mathrm{d}}

\def\GL{\mathrm{GL}}
\def\SL{\mathrm{SL}}

\def\Hba{\overline{\cH}_{a,m}}
\def\Hta{\widetilde{\cH}_{a,m}}
\def\Hb1{\overline{\cH}_{m}}
\def\Ht1{\widetilde{\cH}_{m}}

\def\flp#1{{\left\langle#1\right\rangle}_p}
\def\flm#1{{\left\langle#1\right\rangle}_m}
\def\dmod#1#2{\left\|#1\right\|_{#2}}
\def\dmodq#1{\left\|#1\right\|_q}

\newcommand{\funccol}{\colon \thinspace}
\setlength\parindent{0pt}

\def\Zm{\Z/m\Z}

\def\Err{{\mathbf{E}}}

\newcommand{\comm}[1]{\marginpar{%
\vskip-\baselineskip 
\raggedright\footnotesize
\itshape\hrule\smallskip#1\par\smallskip\hrule}}

\def\xxx{\vskip5pt\hrule\vskip5pt}



\title[Solutions to systems of polynomial  congruences in balls]{Distribution of solutions to systems of congruences in balls}
\author{Michael Harm}
\address{School of Mathematics and Statistics, UNSW, 
Sydney, NSW 2052, Australia}
\email{m.harm@unsw.edu.au}


\begin{abstract}
    Let $G_1,\dots, G_n\in \F_p[X_1,\dots,X_m]$ be $n$ polynomials in $m$ variables over the finite field $\F_p$ of $p$ elements. For any sufficiently large prime $p$ and non-trivial bounds for the Weyl sums associated to the non-trivial linear combinations of $G=(G_1,\dots, G_n)$, we study various properties regarding the distribution of the vectors of fractional parts
    \begin{equation*}
        \bigg(\bigg\{ \frac{G_1(\textbf{x})}{p}\bigg\},\cdots,\bigg\{ \frac{G_n(\textbf{x})}{p}\bigg\}\bigg)\in \T^n,\hspace{10pt} \textbf{x}\in \F_p^m.
    \end{equation*}
    We prove refinements of equidistribution, such as bounds for the ball discrepancy and variance. 
\end{abstract}
\maketitle
\thispagestyle{empty}
    
\section{Introduction}

\subsection{Equidistribution}
    Let $p$ be a prime and let $\F_p$ be the finite field of $p$ elements. We define a polynomial system $G=(G_1,\dots ,G_n)$ with $n\geq 2$ polynomials $G_j$ in $m$ variables over $\F_p$:
    \begin{equation*}
        G_j(X_1,...,X_m)\in \F_p[X_1,...,X_m],\thinspace j=1,...,n.
    \end{equation*}
    
    Dividing $G$ by $p$ and taking fractional parts then induces a map onto the Torus $\T^n:=(\R/\Z)^n$ and its image yields the multiset
    \begin{equation}\label{eq: multiset}
        A_G:=\bigg\{ \bigg( \frac{G_1(x)}{p},\dots, \frac{G_n(x)}{p}\bigg) \funccol x\in \F_p^m \bigg\}\subset \T^n.
    \end{equation}
    This set contains $\#A_G=p^m$ points (counted with multiplicities). 
    \begin{definition}\label{def: measure}
        Let $B\subset\T^n$ be a Borel set. We assign the probability measure $\mu_G$ to the multiset $A_G$ on $\T^n$ by
    \begin{equation}\label{eq: mu G}
        \mu_G(B) := \frac{\#(B\cap A_G)}{p^m}.
    \end{equation}
    
    Let $f\funccol \T^n \rightarrow \C$ be a measurable function. We denote the integral with respect to the measure $\mu_G$ by
    \begin{equation}\label{eq: integral mu G}
        \int_{\T^n} f(x) \diff \mu_G(x) := \frac{1}{p^m} \sum_{d\in A_G}f(d) .
    \end{equation}
    \end{definition}
    In \Cref{sec: fourier transform} we will use the Fourier transform to write \eqref{eq: integral mu G} in terms of the following Weyl sum.
    \begin{definition}\label{def: Weyl sum}
        Let $v\in\Z^n$. We denote the Weyl sum associated to $G$ by
        \begin{equation}\label{eq: Weyl sum}
            S_G(v)= \sum_{x\in \F_p^m} \ep\bigg(\sum_{j=1}^n v_j G_j(x)\bigg),
        \end{equation}
        where we denote $\ep(r)=e^{2\pi i r/p}$ for $r\in\R$.
    \end{definition}
    Our results will largely depend on having a good upper bound for \eqref{eq: Weyl sum} associated to a given polynomial system $G$. Noticing that $p^m$ is a trivial upper bound, we categorize $G$ as follows.
    \begin{definition}\label{def: eta}
        We say a polynomial system $G$ is of type $\eta\geq 0$, if 
        \begin{equation*}
            \frac{|S_G(v)|}{p^m}\ll\begin{cases}
                1 & \text{if }v\in p\Z^n,\\
                p^{-\eta} & \text{otherwise,}
            \end{cases}
        \end{equation*}
        where $\ll$ denotes the Vinogradov symbol, namely $f\ll g$ is equivalent to $|f|\leq C g$ for some positive, real $C$ and functions $f,g$, where $g$ attains non-negative, real values.
    \end{definition}
    \begin{example}
        We can adapt various existing bounds for Weyl sums, such that they hold for any non-trivial linear combination of $G$. 
        \begin{enumerate}
            \item For  arbitrary non-constant polynomials modulo $p$ we may use the classical Weil 
            bound, see, e.g.,~\cite[Chapter~6, Theorem~3]{li1996number} 
            to achieve $\eta=1/2$.
            \item Let $G_{v,d(v)}$ be the degree $d(v)$ homogeneous part of the degree $d(v)$ polynomial $v\cdot G$. If the locus $G_{v,d(v)}=0$ is a non-singular hypersurface in $\P^{m-1}$ for all non-trivial linear combinations $v\cdot G$, then we may use Deligne's result \cite[Theorem 8.4]{deligne1974conjecture} to achieve $\eta=m/2$.
            \item We refer e.g. to \cite{fouvry2001general} for a catalogue of varying bounds for exponential sums.
        \end{enumerate}
    \end{example}
    Note that while it is generally hard to preserve the non-singularity condition of the individual polynomials $G_j$ for all non-trivial linear combinations, we may simply require the polynomials to be of different degrees. The measures $\mu_G$ (Definition \ref{def: measure}) equidistribute for all polynomial systems $G$ of positive type $\eta>0$ as $p\rightarrow\infty$. Thus for every continuity set $B\subset\T^n$ we have
        \begin{equation}\label{eq: equidistribution measure}
            \lim_{p\rightarrow\infty}\mu_G (B)= \mu (B),
        \end{equation}
        and for every continuous function $f\funccol \T^n\rightarrow\C$ we have
        \begin{equation}\label{eq: equidistribution integral}
            \lim_{p\rightarrow\infty}\int_{\T^n}f(x) \diff \mu_G(x) = \int_{\T^n}f(x) \diff x.
        \end{equation}

    The proof of \eqref{eq: equidistribution integral} follows directly from Lemma \ref{lemma: Fourier transform} and \eqref{eq: equidistribution measure} then follows by inserting the indicator function of the set $B$ and the Portmanteau theorem. In this paper we are studying various refinements of this equidistribution result. We follow the methods of Humphries \cite{humphries_2021} who used the Weil bound for Kloosterman sums \cite[Corollary 11.12]{iwaniec2021analytic} to prove the uniform distribution of the points $G(x)=(x,x^{-1})\in \F_q^2$, for sufficiently large integers $q$ by bounding its ball discrepancy as $q\rightarrow \infty$. Notably, Humphries samples over $x\in\F_q^*$ from the multiplicative group modulo $q$ instead of the entire ring. Furthermore, the function $G(x)\equiv (x,x^{\varphi(q)-1}) \text{ mod }(q)$ can be written as a polynomial system via Euler's theorem (Here $\varphi$ denotes the Euler totient function). Kerr and Shparlinski \cite{kerr2012distribution} used a variation of Deligne's result \cite[Equation 10.6]{fouvry2001general} to prove the uniform distribution of the points $G(x)\in \F_p^n$, where $G$ is a degree 2 independent polynomial system and $x$ runs through a rather general subset of $\Gamma\subseteq \F_p^m$. We aim to generalize these results for any polynomial system $G$ with $n\geq 2$ polynomials whose associated Weyl sums (\Cref{def: Weyl sum}) can be bounded by $p^{m-\eta}$ (\Cref{def: eta}).
    
\subsection{Discrepancy}    
    We start by bounding the difference of the measures $\mu_G$ and the Lebesgue measure $\mu$ over any injective, geodesic $n$-ball in $T^n$.
    
    \begin{theorem}\label{thm: shrinking target}
        Let $G$ be a polynomial system of type $\eta\geq 0$ (Definition \ref{def: eta}). Furthermore let $B_R(y)$ be a Euclidean $n$-ball with radius $0< R<1/2$ and center point $y\in \T^n$ (if given). We have
        \begin{equation*}
        \begin{split}
            |\mu_G(B_R(y))-\mu(B_R)| \ll &\begin{cases}
                R^{\frac{n(n-1)}{n+1}}p^{-\frac{2\min (\eta,n)}{n+1}} &\text{for } R>p^{-\frac{\min (\eta,n)}{n}},\\
                p^{-\min(\eta,n)} &\text{for } R\leq p^{-\frac{\min (\eta,n)}{n}},
            \end{cases} 
        \end{split}
        \end{equation*}
        where $\mu$ denotes the Lebesgue measure.
    \end{theorem}

    Taking the classical Weil-bound $\eta=\frac{1}{2}$, this result may also be computed by the methods of Harman \cite{Har98}. Our next refinement is bounding the ball discrepancy of the measures $\mu_G$ as $p\rightarrow\infty$.
    
    \begin{definition}\label{def: ball discrepancy}
        Assume the same setting as in Theorem \ref{thm: shrinking target}. We define the ball discrepancy of the measures $\mu_G$ by
        \begin{equation}\label{eq: ball discrepancy}
            D(\mu_G):= \sup_{\substack{y\in\T^n\\0<R<\frac{1}{2}}}|\mu_G(B_R(y))-\mu(B_R)|.
        \end{equation}
    \end{definition}
    Note that the ball discrepancy differs from the box discrepancy $D^\text{box}(\mu_G)$, where we take the supremum over all boxes $[a_1,b_1]\times \cdots \times [a_n,b_n]$. In fact, one can show $D(\mu_G)\ll_n D^\text{box}(\mu_G)^{1/n}$ via \cite[Chapter 2, Theorem 1.6]{kuipers1974uniform}.
     
    \begin{corollary}\label{theorem: Discrepancy}
        As $p \rightarrow \infty$, the ball discrepancy (\Cref{def: ball discrepancy}) of the measures $\mu_G$ satisfies
            \begin{equation}\label{eq:ball discrepancy mu}
                D(\mu_G)\ll p^{-2\min (\eta,n)/(n+1)}.
            \end{equation}        
    \end{corollary}
    These are generalizations of \cite[Theorems 1.1(1) \& 1.2(1)]{humphries_2021}, as inserting $n=2$ and the Weil bound for Kloosterman sums $\eta=1/2$ yields the same results. Theorem \ref{thm: shrinking target} is notably the dual to the paper by Kerr and Shparlinski \cite{kerr2012distribution}, who restrict the domain of the measures $\mu_G$ to rather general subsets of $\T^m$. In contrast, in this paper we restrict the codomain $\T^n$ to Euclidean balls.
    
\subsection{Variance}
    Our other refinement concerns the variance of the injective, geodesic $n$-balls of radius $R$ as $p\rightarrow\infty$.
    \begin{definition}\label{def: variance}
        Let $G$ be a polynomial system with measure $\mu_G$ (Definition \ref{def: measure}). The \mbox{variance} of $\mu_G$ with respect to a Borel set $B\subset \T^n$ is given by
        \begin{equation*}
            \text{Var}(\mu_G, B) := \int_{\T^n} (\mu_G(B(y))-\mu(B))^2 \thinspace \diff y,
        \end{equation*}
        where $B(y)\subset \T^n$ denotes the translation of the set $B$ by a vector $y\in\T^n$.
    \end{definition}

    \begin{theorem}\label{thm: variance}
        Let $G$ be a polynomial system of type $\eta$ (Definition \ref{def: eta}) and let $B_R(y)\subset \T^n$ denote the $n$-ball of radius $R<1/2$ with center $y\in\T^n$. The variance satisfies
        \begin{equation*}
            \text{Var}(\mu_G,B_R) \ll \frac{R^n}{p^{\min(2\eta,n)}}.
        \end{equation*}
    \end{theorem}
\subsection{Proof strategy}
    We express averages with respect to the probability measure $\mu_G$ using Fourier expansions of the indicator functions of Euclidean balls and the Weyl sums associated to a polynomial system $G$. For the proof of Theorem \ref{thm: shrinking target} we introduce smooth majorants and minorants of ball indicators, in order to ensure absolute convergence of the Fourier series. This is not necessary for the proof of Theorem \ref{thm: variance}, as the Fourier series for the variance case is already absolutely convergent. For both Theorems \ref{thm: shrinking target} \& \ref{thm: variance} we then split the frequency domain such that Weyl sum bounds of type $\eta$ can be applied effectively to control the error terms.
\section{Preliminaries}

\subsection{Fourier transform}\label{sec: fourier transform}
    In this section we define the Fourier transform on the Torus $\T^n$ and use it to obtain an identity that expresses \eqref{eq: integral mu G} in terms of the Weyl sums $S_G(v)$ (see \eqref{eq: Weyl sum}).
    \begin{definition}\label{def: Fourier transform}
        Let $f\in L^1(\T^n)$ be an integrable function. We define the Fourier transform of $f$ by
        \begin{equation}\label{eq: fourier transform}
            \widehat{f}(v) := \int_{\T^n} f(x) \textbf{e}(-v\cdot x)\diff x, \text{ for all }v\in \Z^n,
        \end{equation}
        where we denote $\e(r)=\e_1(r)$ for all $r\in\R$, and by $v\cdot w=\sum_{i=1}^n v_i w_i$ the inner product for $v,w\in \R^n$.  The corresponding Fourier series (see \cite[Chapter 3]{loukas2014classical}) is given by 
        \begin{equation}\label{eq: fourier series}
            f(x)= \sum_{v\in\Z^n} \widehat{f}(v) \textbf{e}(v\cdot x), \text{ for all }x\in \T^n.
        \end{equation}
    \end{definition}
    
    \begin{lemma}\label{lemma: Fourier transform}
        Let $f\in L^1(\T^n)$ be an integrable function and let $\widehat{f}$ be its Fourier transform \eqref{eq: fourier transform}. We have the identity for the integral associated to the measure $\mu_G$ (\Cref{def: measure})
        \begin{equation*}
            \int_{\T^n} f(x) \diff \mu_G(x) =\frac{1}{p^m}\sum_{v\in\Z^n} \widehat{f}(v) S_G(v),
        \end{equation*}
        where $S_G(v)$ is the Weyl sum \eqref{eq: Weyl sum}.
    \end{lemma}
    \begin{proof}
        We insert \eqref{eq: fourier series} and then exchange the order of integration and summation to obtain
    \begin{equation*}
    \begin{split}
        \int_{\T^n} f(x) \diff \mu_G(x)&=\int_{\T^n} \sum_{v\in\Z^n} \widehat{f}(v) \textbf{e}(v\cdot x) \text{ d}\mu_G(x)\\
        &=\sum_{v\in\Z^n} \widehat{f}(v)\int_{\T^n}  \textbf{e}(v\cdot x) \text{ d}\mu_G(x)\\
        &=\frac{1}{p^m}\sum_{v\in\Z^n} \widehat{f}(v) S_G(v).
    \end{split}
    \end{equation*}
    \end{proof}

\subsection{A Bessel function of the first order}
    In this section we give a brief overview over a Bessel function of the first kind. For further reading we reference to \cite{korenevbessel}. For the purposes of this paper, it suffices to define a Bessel function of the first kind by its infinite series (see e.g. \cite[Page 9, Equation (1.7)]{korenevbessel}).
    \begin{definition}
        We define the Bessel function of order $\nu\in\R$ of the first kind by
        \begin{equation}\label{eq: bessel}
            J_\nu(x):=\bigg(\frac{x}{2}\bigg)^\nu \sum_{k=0}^\infty\frac{(-1)^{k}}{k!\thinspace\Gamma(\nu+k+1)}\bigg(\frac{x}{2}\bigg)^{2k},
        \end{equation}
        for all $x\in\R$, where $\Gamma(\cdot)$ denotes the usual Gamma function (see e.g. \cite[Chapter 7.2]{makarovanalysis}).
    \end{definition}

    \begin{lemma}\label{lemma: Bessel integral}
    Let $\nu>-1$. We have
        \begin{equation*}
            \int_0^1 x^{\nu+1} J_\nu(ax) \thinspace \diff x = a^{-1} J_{\nu+1}(a),
        \end{equation*}
        for all $a\in \R$.
    \end{lemma}
    \begin{proof}
        From Sonine's first finite integral \cite[Page 65, Equation (20.1)]{korenevbessel} we obtain
        \begin{equation*}
            a^{-1}J_{\nu+1}(a) = \int_{0}^{\pi/2} J_\nu(a \sin\theta) \sin^{\nu+1}\theta \cos \theta \diff \theta.
        \end{equation*}
        Making a change of variables $\sin \theta \rightarrow x$ then gives us the desired result.
    \end{proof}
    \begin{lemma}\label{lemma: Bessel bound}
        Let $x\geq 0$. We have the following bound for the Bessel function of order $\nu$ of the first kind
        \begin{equation*}
            J_\nu(x)\ll \min\bigg\{x^\nu , \frac{1}{\sqrt{x}}\bigg\}.
        \end{equation*}
    \end{lemma}
    \begin{proof}
        As $x\rightarrow 0$ we obtain the first bound directly from \eqref{eq: bessel}. For large values of $x$ the second bound is a direct consequence of \cite[Page 127, Equation (29.4)]{korenevbessel}.
    \end{proof}

\subsection{Indicator function of \texorpdfstring{$B_R\subset\T^n$}{Lg} }

    In this section we formalize the indicator function of the $n$-ball $B_R$ on the Torus $\T^n$.     If we then insert said indicator function for $f$ in \eqref{eq: integral mu G}, this becomes $\mu_G$. Thus we may calculate the ball discrepancy \eqref{eq: ball discrepancy} of $\mu_G$ using the Fourier transform of the indicator function.
    \begin{definition}
        Let $n\geq 2$ and $0<R$. We define the point-pair invariant $k_R\funccol \R^n \times \R^n \rightarrow \R$ by
        \begin{equation}\label{eq: point pair invariant}
            k_R(x,y):=\begin{cases}
            1 & \text{if }\|x-y\|\leq R,\\
            0 & \text{otherwise,}
        \end{cases}
        \end{equation}
        where we denote the Euclidean norm by $\|v\|=\|v\|_2=\sqrt{v\cdot v}$. This induces a point-pair invariant $K_R\funccol \T^n \times \T^n \rightarrow \R$ given by
        \begin{equation}\label{eq: n ball indicator}
            K_R(x,y):= \sum_{v\in \Z^n} k_R(x+v,y).
        \end{equation}
        If $R<1/2$, $K_R$ becomes the indicator function of the injective geodesic $n$-Ball $B_R(y)$ as a function in $x\in\T^n$.
    \end{definition}

    \begin{lemma}\label{lemma: Fourier transform indicator}
        Take $K_R(x,y)$ as in \eqref{eq: n ball indicator}. The Fourier transform of $K_R$ in the variable $x$ is given by
        \begin{equation*}
        \widehat{K}_R(v,y) = \begin{cases}
            \mu(B_R)    & \text{if } v=0,\\
            \frac{R^{n/2}J_{n/2}(2\pi R \|v\|)}{\|v\|^{n/2 }} \e(-v\cdot y) & \text{otherwise},
        \end{cases} 
        \end{equation*}
        for $v\in \Z^n$.
    \end{lemma}

    \begin{proof}
        By \Cref{def: Fourier transform} we may calculate the Fourier coefficients of $K_R$ as follows:
        \begin{equation*}
        \begin{split}
            \widehat{K}_R(v,y) :=& \int_{\T^n} K_R(x,y)\e(-v\cdot x) \diff x\\
            =& \int_{\R^n} k_R(x,y)\e(-v\cdot x)\diff x .
        \end{split}
        \end{equation*}
        Here we make a change in variables $x\rightarrow Rx+y$, so that $k$ becomes the indicator function of the unit ball around the origin. Inserting \eqref{eq: point pair invariant} then gives us
        \begin{equation}\label{eq: radial function}
            R^n \e(-v\cdot y) \int_{B_1(0)} \e(-R v\cdot x) \diff x.
        \end{equation}
        Since the indicator function of the unit ball around the origin is a radial function, we may use \cite[Chapter IV, Theorem 3.3]{stein1971introduction}. Thus \eqref{eq: radial function} becomes
        \begin{equation*}
            2\pi R^{(n+2)/2} \e(-v\cdot y) \|v\|^{(2-n)/2 }\int_0^1 J_{(n-2)/2}(2\pi R \|v\| r) r^{n/2} \diff r.
        \end{equation*}
        Hence Lemma \ref{lemma: Bessel integral} gives us
        \begin{equation*}
            \widehat{K}_R(v,y)= \frac{R^{n/2}J_{n/2}(2\pi R \|v\|)}{\|v\|^{n/2 }} \e(-v\cdot y). 
        \end{equation*}
        The term for $v=0$, may be directly calculated as follows
        \begin{equation}\label{eq:zero term}
        \begin{split}
            \widehat{K}_R(0,y) :=& \int_{\T^n} K_R(x,y)\diff x\\
            =& \int_{\R^n} k_R(x,y)\diff x \\
            =&\mu(B_R).
        \end{split}
        \end{equation}
        This concludes the proof of \Cref{lemma: Fourier transform indicator}.
    \end{proof}
    
    The Fourier series \eqref{eq: fourier series} for $K_R$ does not converge absolutely, since the Fourier coefficients are not of sufficiently rapid decay. To work around this, we define mayorants and minorants.

    \begin{definition}
        Let $n\geq 2$, $0<\rho<R$ and $k_R$ as in \eqref{eq: point pair invariant}. We define the convolutions
        \begin{equation}\label{eq: point pair invariant pm}
        \begin{split}
            k_{R,\rho}^\pm (x,y) :=&\frac{1}{\mu(B_\rho)}k_{R\pm\rho}*k_\rho (x,y)\\
            =&\frac{1}{\mu(B_\rho)}\int_{\R^n}k_{R\pm\rho}(x,w)k_\rho (w,y)dw.
        \end{split}
        \end{equation}
        These induce the point-pair invariants $K_R^\pm\funccol \T^n \times \T^n \rightarrow \R$ given by
        \begin{equation}\label{eq: n ball indicator pm}
            K_{R,\rho}^\pm(x,y):= \sum_{v\in \Z^n} k_{R,\rho}^\pm(x+v,y).
        \end{equation}
    \end{definition}
    
    It is readily checked that $k_{R,\rho}^\pm(x,y)$ are both non-negative, continuous, pointwise linear in radial coordinates, bounded by 1 and satisfy
    \begin{equation}\label{eq: pointwise inequality plus}
        k_{R,\rho}^+(x,y)=\begin{cases}
            1 & \text{if } |x-y|\leq R,\\
            0 & \text{if } |x-y|> R+2\rho,
        \end{cases}
    \end{equation}
    and
    \begin{equation}\label{eq: pointwise inequality minus}
        k_{R,\rho}^-(x,y)=\begin{cases}
            1 & \text{if } |x-y|\leq R-2\rho,\\
            0 & \text{if } |x-y|> R.
        \end{cases}
    \end{equation}
    
    \begin{lemma}\label{lemma: indicator inequalities}
        For all $x,y\in \T^n$ we have the pointwise inequalities
        \begin{equation*}
            K^-_{R,\rho}(x,y)\leq K_R(x,y)\leq K^+_{R,\rho}(x,y).
        \end{equation*}
    \end{lemma}
    \begin{proof}
        This follows immediately from \eqref{eq: pointwise inequality plus} and \eqref{eq: pointwise inequality minus}.
    \end{proof}

    \begin{lemma}\label{lemma: Fourier indicator pm}
        The Fourier transform of $K_{R,\rho}^\pm$ is given by
        \begin{equation*}
        \begin{split}
            \widehat{K}_{R,\rho}^\pm(v,y) &= \begin{cases}
            \mu(B_{R\pm\rho})    & \text{if } v=0,\\
            \frac{(R\pm\rho)^{n/2}\rho^{n/2}}{\mu(B_\rho)} \frac{J_{n/2}(2\pi (R\pm \rho) \|v\|)J_{n/2}(2\pi \rho \|v\|)}{\|v\|^{n}} \e(-v \cdot y) & \text{otherwise},
            \end{cases} 
        \end{split}
        \end{equation*}
        for $v\in \Z^n$.
    \end{lemma}
    \begin{proof} 
        We insert \eqref{eq: point pair invariant pm} into \eqref{eq: n ball indicator pm} and take the Fourier transform \eqref{eq: fourier transform} to obtain
        \begin{equation*}
        \begin{split}
            \widehat{K}_{R,\rho}^\pm(v,y)&=\frac{1}{\mu(B_\rho)} \int_{\R^n} \int_{\R^n}k_{R\pm\rho}(x,w)k_\rho (w,y)\diff w \textbf{ }\e(-v\cdot x) \diff x.
        \end{split}
        \end{equation*}
        As a consequence of \eqref{eq: pointwise inequality plus} we immediately obtain the bound ${\widehat{K}_{R,\rho}^\pm (v,y)\leq \mu(B_{R+2\rho})<\infty}$. Therefore we may use the Fubini-Tonelli theorem \cite[Proposition 5.2.1]{cohn2013measure} to interchange the order of integration.
        \begin{equation*}
        \begin{split}
            &\frac{1}{\mu(B_\rho)} \int_{\R^n}\bigg( \int_{\R^n}k_{R\pm\rho}(x-w,0) \e(-v\cdot x) \diff x\bigg)  \textbf{ }k_\rho (w,y)\diff w\\
            =&\frac{1}{\mu(B_\rho)}\widehat{K}_{R\pm\rho} (v,0) \int_{\R^n}  \e(-v\cdot w)  \textbf{ }k_\rho (w,y)\diff w\\
            =&\frac{1}{\mu(B_\rho)}\widehat{K}_{R\pm\rho} (v,0) \widehat{K_{\rho} }(v,y) .
        \end{split}
        \end{equation*}
        Finally we obtain the result via \Cref{lemma: Fourier transform indicator}.
    \end{proof}
    
    The Fourier series \eqref{eq: fourier series} for $K^\pm_{R,\rho}$ therefore converge absolutely. This can be seen by inserting the bound in Lemma \ref{lemma: Bessel bound} and then applying Lemma \ref{lemma: Zeta value} below.

\subsection{A Zeta function}
    Let $n\geq 2$. In \Cref{sec: proofs} we will need to evaluate Zeta functions of the form
    \begin{equation}\label{eq: zeta nab}
        \zeta_{n,a,b}(s) =\sum_{\substack{v\in\Z^n\\a<\|v\|\leq b}}\frac{1}{\|v\|^s},
    \end{equation}
    for $s\in\R$, where $0\leq a\leq b$. For the evaluation of \eqref{eq: zeta nab} it proves useful to group the terms $\|v\|^{-s}$ together that have the same value. Incidentally these are exactly the integer points on the $(n-1)$-sphere of radius $\|v\|$.
    
    \begin{definition}\label{def: sphere points}
        Let $k,K\in \N$. We denote the sum of squares function, or equivalently, the number of lattice points on an $n-1$-sphere of radius $\sqrt{k}$, by
        \begin{equation*}
            r_n(k):=\#\{v\in\Z^n : \|v\|^2 = k \}.
        \end{equation*}
        We denote the number of lattice points within an $n$-ball of radius $\sqrt{K}$ by
        \begin{equation*}
        \begin{split}
            \cR_n(K) :=&\#\{v\in\Z^n : \|v\|^2 \leq K \} =\sum_{k=0}^K r_n(k).
        \end{split}
        \end{equation*}
    \end{definition}

    \begin{lemma}\label{lemma: Zeta value}
        Take the zeta function $\zeta_{n,a,b}(s)$ defined in \eqref{eq: zeta nab}. We obtain the following bound
        \begin{equation*}
            \zeta_{n,a,b}(s) \ll b^{n-s}+a^{n-s}.
        \end{equation*}
    \end{lemma}
    
    \begin{proof}
        First we use \Cref{def: sphere points} to group summands of the same value together.
        \begin{equation*}
        \begin{split}
            \zeta_{n,a,b}(s)= &\sum_{\substack{v\in \Z^n\\ a^2 <\|v\|^2 \leq b^2}}\frac{1 }{(\|v\|^2)^{s/2}}= \sum_{\substack{k\in\N\\ a^2 < k \leq b^2}}\frac{r_n(k) }{k^{s/2}}.
        \end{split}
        \end{equation*}
        Here we may use Abel's identity (see e.g. \cite[Theorem 4.2]{apostol2013introduction}) to obtain
        \begin{equation*}
             \zeta_{n,a,b}(s)=\cR_n(b^2) b^{-s} - \cR_n(a^2) a^{-s}  + \frac{s}{2}\int_{a^2}^{b^2} \frac{\cR_n(x)}{ x^{s/2+1}} \diff x.
        \end{equation*}
        The lattice points within an $n$-ball of radius $\sqrt{K}$ can be bounded by its volume $\cR(K)\ll K^{n/2}$ (see e.g.\cite[Page 632]{makarovanalysis}). Therefore we obtain
        \begin{equation}\label{eq:D2 insignificant}
        \begin{split}
            \zeta_{n,a,b}(s)\ll &b^{n-s} + a^{n-s}  + \int_{a^2}^{b^2} x^{(n-s-2)/2} \diff x \\
            \ll& b^{n-s}+a^{n-s} .
        \end{split}
        \end{equation}
    \end{proof}

    \begin{lemma}\label{lemma: Zeta Weyl sum}
        Let $0\leq a\leq b$ and $s\in\R$. Furthermore let $c>0$ and assume $G$ is of type $\eta$ (\Cref{def: eta}). We have
        \begin{equation}\label{eq: Zeta Weyl sum}
            \sum_{\substack{v\in\Z^n\\a<\|v\|\leq b}}\frac{|S_G(v)|^c}{p^{cm}\|v\|^s}\ll (p^{-c\eta}+p^{-n})  (b^{n-s}+a^{n-s}).
        \end{equation}
    \end{lemma}
    \begin{proof}
        By \Cref{def: eta} we may split up \eqref{eq: Zeta Weyl sum} as follows
        \begin{equation*}
        \begin{split}
            \sum_{\substack{v\in\Z^n\\a<\|v\|\leq b}}\frac{|S_G(v)|^c}{p^{cm}\|v\|^s}\ll& p^{-c\eta}\sum_{\substack{v\in\Z^n\\a<\|v\|\leq b}}\frac{1}{\|v\|^s} +p^{-s}\sum_{\substack{v\in \Z^n\\a/p<\|v\|\leq b/p}}\frac{1}{\|v\|^s}.
        \end{split}
        \end{equation*}
        Thus we may use \Cref{lemma: Zeta value} to obtain the bound
        \begin{equation*}
            (p^{-c\eta}+p^{-n})  (b^{n-s}+a^{n-s}).
        \end{equation*}
    \end{proof}

\section{Proofs}\label{sec: proofs}

\subsection{Proof of Theorem \ref{thm: shrinking target}}
    Since the indicator function of $B_R(y)$ is given by \eqref{eq: n ball indicator}, we may write
    \begin{equation}\label{eq: measure in terms of indicator}
        \mu_G(B_R(y))= \int_{\T^n}K_R(x,y) d\mu_G(x).
    \end{equation}
    For any $0<\rho<R$, \Cref{lemma: indicator inequalities} gives us
    \begin{equation*}
        \int_{\T^n}K_{R,\rho}^-(x,y) \diff \mu_G(x) \leq \mu_G(B_R(y)) \leq \int_{\T^n}K_{R,\rho}^+(x,y) \diff \mu_G(x).
    \end{equation*}
    We then derive the absolutely convergent spectral expansion with \Cref{lemma: Fourier transform} and insert the Fourier coefficients as in \Cref{lemma: Fourier indicator pm}.
    \begin{multline}\label{eq:spectral expansion}
        \int_{\T^n}K_{R,\rho}^\pm(x,y) \diff \mu_G(x) = \mu(B_{R\pm\rho}) \\
        + \frac{(R\pm \rho)^{n/2}\Gamma(\frac{n}{2}+1)}{\pi^{n/2}\rho^{n/2} \cdot p^m } \sum_{v\in \Z^n\setminus \{0\}} \frac{J_{n/2}(2\pi (R\pm \rho) \|v\|)\cdot J_{n/2}(2\pi \rho \|v\|)}{\|v\|^n} \e(-v \cdot y) S_G(v).
    \end{multline}
    To simplify calculations, we assume $\rho\leq R/2$. We may partition the sum in \eqref{eq:spectral expansion}, such that each of the two $J$-Bessel functions attain their respective minima given by Lemma \ref{lemma: Bessel bound}. Thus

    \begin{equation}\label{eq: indicator pm integral bound 1}
    \begin{split}
        \int_{\T^n}K_{R,\rho}^\pm(x,y) \diff \mu_G(x) \ll& (R\pm\rho)^n \\
        &+ (R\pm \rho)^{n} \sum_{\substack{v\in \Z^n\setminus \{0\}\\0<\|v\|\leq (R\pm \rho)^{-1}}} \frac{|S_G(v)|}{p^m}\\ 
        &+ (R\pm \rho)^{(n-1)/2} \sum_{\substack{v\in \Z^n\setminus \{0\}\\(R\pm \rho)^{-1}<\|v\|\leq \rho^{-1}}} \frac{|S_G(v)|}{ p^m \|v\|^{(n+1)/2}} \\
        &+ \frac{(R\pm \rho)^{(n-1)/2}}{\rho^{(n+1)/2}} \sum_{\substack{v\in \Z^n\setminus \{0\}\\\rho^{-1}<\|v\|}} \frac{|S_G(v)|}{ p^m \|v\|^{n+1}} .
    \end{split}
    \end{equation}
    We can then evaluate the sums in \eqref{eq: indicator pm integral bound 1} by using \Cref{lemma: Zeta Weyl sum}
    \begin{equation*}
        \int_{\T^n}K_{R,\rho}^\pm(x,y) \diff \mu_G(x) \ll (R\pm\rho)^n +p^{-\eta}+p^{-n}+ (R/ \rho)^{(n-1)/2}   (p^{-\eta}+p^{-n}) .
    \end{equation*}
    In order to calculate the ball discrepancy (see \Cref{def: ball discrepancy}), we calculate an upper bound for the discrepancy of the individual $n$-balls dependent on $\rho$ 
    \begin{equation*}
    \begin{split}
        |\mu_G(B_R(y))-\mu(B_R)|  \ll R^{n-1}\rho +  p^{-\min (\eta,n)} +(R/\rho)^{(n-1)/2}   p^{-\min(\eta,n)}.
    \end{split}
    \end{equation*}
    We may then choose $\rho= R^{-\frac{n-1}{n+1}}p^{-\frac{2\min (\eta,n)}{n+1}}$ for $R>p^{-\frac{\min (\eta,n)}{n}}$ and $\rho=R/2$ for $R\leq p^{-\frac{\min (\eta,n)}{n}}$ and conclude that
    \begin{equation}\label{eq: shrinking target bounds}
    \begin{split}
        |\mu_G(B_R(y))-\mu(B_R)| \ll &\begin{cases}
            R^{\frac{n(n-1)}{n+1}}p^{-\frac{2\min (\eta,n)}{n+1}} &\text{for } R>p^{-\frac{\min (\eta,n)}{n}},\\
            p^{-\min(\eta,n)} &\text{for } R\leq p^{-\frac{\min (\eta,n)}{n}}.
        \end{cases} 
    \end{split}
    \end{equation}
    Thus \eqref{eq: shrinking target bounds} gives us the bound of Theorem \ref{thm: shrinking target}. \hfill\qedsymbol{}

\subsection{Proof of Theorem \ref{thm: variance}}
    The variance of $\mu_G$ with respect to the $n$-ball $B_R$ of radius $R<1/2$ is given by 
    \begin{equation}
        \text{Var}(\mu_G,B_R):=\int_{\T^n}(\mu_G(B_R(y))-\mu(B_R))^2\diff y.
    \end{equation}
    Using Parseval's identity (see e.g. \cite[Page 89, Equation (4.94)]{iwaniec2021analytic}) and since $\mu(B_R)$ is constant, we obtain
    \begin{equation*}
        \text{Var}(\mu_G,B_R)=\sum_{v\in\Z^n\setminus\{0\}} \bigg|\widehat{K}_R(v,y) \frac{S_G(v)}{p^{m}}\bigg|^2 .
    \end{equation*}
    \\
    We may now insert the Fourier coefficients via Lemma \ref{lemma: Fourier transform indicator} and obtain the absolutely \mbox{convergent} spectral expansion
    
    \begin{equation*}
    \begin{split}
        \text{Var}(\mu_G,B_R) =& R^{n} \sum_{v\in\Z^n\setminus\{0\}} \frac{J_{n/2}(2\pi R \|v\|)^2 }{\|v\|^{n}}  \frac{|S_G(v)|^2}{p^{2m}}.
    \end{split}
    \end{equation*}
    We then split up the sum via Lemma \ref{lemma: Bessel bound} to obtain the bound
    \begin{equation*}
    \begin{split}
        \text{Var}(\mu_G,B_R)\ll R^{2n} \sum_{\substack{v\in\Z^n\\ 0<\|v\|\leq R^{-1}}}  \frac{|S_G(v)|^2}{p^{2m}} + R^{n-1} \sum_{\substack{v\in\Z^n\\ R^{-1}<\|v\|}}  \frac{|S_G(v)|^2}{p^{2m}\|v\|^{n+1}}.
    \end{split}
    \end{equation*}
    
    Using Lemma \ref{lemma: Zeta Weyl sum} then gives us
    \begin{equation*}
    \begin{split}
        \text{Var}(\mu_G,B_R)\ll& R^{n} (p^{-2\eta}+p^{-n}) .
    \end{split}
    \end{equation*}

    This concludes the proof of Theorem \ref{thm: variance}. \hfill\qedsymbol{}

\section*{Acknowledgements}
    The author thanks their supervisors Igor E. Shparlinski and Bryce Kerr for their valuable input, as well as the anonymous referees for their helpful suggestions. This research has been supported by a University International Postgraduate Award (UIPA, RSRE7061, RSRE7063), a University of New South Wales School of Mathematics scholarship (RSRT6014) and an Australian Research Council Grant (ARC, RSGT002).

\end{document}